\documentclass[10pt,reqno]{amsart}
     \makeatletter
     \def\section{\@startsection{section}{1}%
     \z@{.7\linespacing\@plus\linespacing}{.5\linespacing}%
     {\bfseries%\normalfont\scshape
     \centering
     }}
     \def\@secnumfont{\bfseries}
     \makeatother
\newtheorem{theorem}{Theorem}[section]
\newtheorem{lemma}[theorem]{Lemma}

\newtheorem{corollary}[theorem]{Corollary}
\theoremstyle{definition}
\newtheorem{definition}[theorem]{Definition}
\newtheorem{example}[theorem]{Example}
\theoremstyle{remark}
\newtheorem{remark}[theorem]{Remark}
\numberwithin{equation}{section}
\usepackage{mathrsfs}

\newcommand{\dd}{\,\mathrm{d}}
\newcommand{\sd}{\circ \mathrm{d}}
\newcommand{\DD}{\mathrm{D}}

\newcommand{\R}{\mathbb{R}}

\newcommand{\CC}{\mathscr{C}}
\newcommand{\SP}{\mathscr{S}}

\newcommand{\pr}{\prime}
\newcommand{\gint}{\gamma\text{-}\int}
\newcommand{\1}{\mathbf{1}}

\begin{document}

\title[Existence of L\'evy's area and pathwise integration]{Existence of L\'evy's area and pathwise integration}

\author[Imkeller]{Peter Imkeller}
\address{Peter Imkeller: Institut f\"ur Mathematik, Humboldt-Universit\"at zu Berlin, Berlin, 10099 Berlin, Germany}
\email{imkeller@math.hu-berlin.de}

\author[Pr\"omel]{David J. Pr\"omel}
\address{David J. Pr\"omel: Institut f\"ur Mathematik, Humboldt-Universit\"at zu Berlin, Berlin, 10099 Berlin, Germany}
\email{proemel@math.hu-berlin.de}

\subjclass[2010]{Primary 26A42, 60H05; Secondary 46N30.}

\keywords{F\"ollmer integration; Functional calculus; It\^o formula; L\'evy's area; Rough path; Stratonovich integral.}

\begin{abstract}
  Rough path analysis can be developed using the concept of controlled paths, and with respect to a topology in which L\'evy's area plays a role. For vectors of irregular paths we investigate the relationship between the property of being controlled and the existence of associated L\'evy areas. For two paths, one of which is controlled by the other, a pathwise construction of the L\'evy area and therefore of mutual stochastic integrals is possible. If the existence of quadratic variation along a sequence of partitions is guaranteed, this leads us to a study of the pathwise change of variable (It\^o) formula in the spirit of F\"ollmer, from the perspective of controlled paths.
\end{abstract}

\maketitle

\section{Introduction}

The theory of rough paths (see \cite{Lyons2007,Lejay2009,Friz2013}) has established an analytical frame in which stochastic differential and integral calculus beyond Young's classical notions is traced back to properties of the trajectories of processes involved without reference to a particular probability measure. For instance, in the simplest non-trivial setting it provides a topology on the set of continuous functions enhanced with an ``area'', with respect to which the (It\^o) map associating the trajectories of a solution process of a stochastic differential equation driven by trajectories of a continuous martingale is continuous. In this topology, convergence of a sequence of functions $X^n = (X^{1,n},\dots,X^{d,n})_{n\in\mathbb{N}}$ defined on the time interval $[0,T]$ involves besides uniform convergence also the convergence of the L\'evy areas associated to the vector of trajectories, formally given by
\begin{equation*}
  \mathbb{L}^{i,j,n}_t := \int_0^t (X_s^{i,n} \dd X_s^{j,n} - X_s^{j,n} \dd X_s^{i,n}),\quad 1\le i,j\le d,\quad t\in[0,T].
\end{equation*}

In probability theory the concept of L\'evy's area was already studied in the 1940s. It was first introduced by P. L\'evy in \cite{Levy1940} for a two dimensional Brownian motion $(B^1,B^2)$. For time $T$ fixed and any trajectory of the process it is defined as the area enclosed by the trajectory $(B^1, B^2)$ and the chord given by the straight line from $(0,0)$ to $(B^1_T, B^2_T)$, and may be expressed formally by
\begin{equation*}
  \frac{1}{2}\bigg( \int_0^T B_t^1 \dd B_t^2 - \int_0^T B_t^2 \dd B_t^1\bigg ),
\end{equation*}
provided the integrals make sense.

More recently, an alternative calculus with a more Fourier analytic touch has been designed (see \cite{Gubinelli2014, Perkowski2013a}) in which an older idea by Gubinelli \cite{Gubinelli2004} is further developed. It is based on the concept of \emph{controlled paths}. In this calculus, rough path integrals are described in terms of Fourier series for instance in the Haar-Schauder wavelet, and are seen to decompose into different parts, one of them representing L\'evy's area. The existence of a stochastic integral in this approach is seen to be linked to the existence of the corresponding L\'evy area, and both can be approximated along a Schauder development in which H\"older functions are limits of their finite degree Schauder expansions. In its simplest (one-dimensional) form a path of bounded variation $Y$ on $[0,T]$ is \emph{controlled by} another path $X$ of bounded variation on $[0,T]$, if the associated signed measures $\mu_X, \mu_Y$ on the Borel sets of $[0,T]$ satisfy that $\mu_Y$ is absolutely continuous with respect to $\mu_X$. In its version relevant here two rough (vector valued) functions $X$ and $Y$ on $[0,T]$ are considered, both with finite $p$-variation for some $p\ge 1$. In the simplest setting, $Y$ is \emph{controlled by} $X$ if there exists a function $Y'$ of finite $p$-variation such that the first order Taylor expansion errors
\begin{equation*}
  R^Y_{s,t} = Y_t-Y_s - Y'_s (X_t-X_s)
\end{equation*}
are bounded in a suitable semi-norm, i.e. $\sum_{[s,t]\in\pi} |R^Y_{s,t}|^r$ is bounded over all possible partitions $\pi$ of $[0,T]$. Here $\frac{1}{r} = \frac{2}{p}.$ Since for a path $X$ H\"older continuity of order $\frac{1}{p}$ is closely related to finite $p$-variation, the control relation can be seen as expressing a type of fractional Taylor expansion of first order: the first order Taylor expansion error of $Y$ with respect to $X$ - both of H\"older order $\frac{1}{p}$ and ``derivative'' $Y'$ - is of double H\"older order $\frac{2}{p}.$ In its \emph{para-controlled} refinement as developed by Gubinelli et al. in \cite{Gubinelli2013} this notion has been seen to give an alternative approach to classical rough path analysis and is suitable for the application to singular PDEs. In the comparison of the two approaches, to make the It\^o map continuous, information stored in the L\'evy areas of vector valued paths has to be complemented by information conveyed by path control or vice versa. This raises the problem about the relationship between the existence of L\'evy's area and the control relationship between vector trajectories or the components of such. We shall deal with this fundamental problem in Section \ref{sec:levy}.

Based on this study we then decompose Riemann approximations of different versions of integrals into a symmetric and an antisymmetric component and prove that for the classical Stratonovich integral just the antisymmetric Riemann sums have to converge, while for more general Stratonovich or It\^o type integrals the existence of limits for the symmetric part has to be guaranteed along fixed sequences of partitions, as in F\"ollmer's approach \cite{Follmer1979}. Under this assumption we additionally derive a pathwise version of a functional It\^o formula due to \cite{Ahn1997}, where the functional has to be just defined on the space of continuous functions. At this point our It\^o formula circumvents a technical problem of Dupire differentiability (see \cite{Dupire2009,Cont2013}), where the functional has to be defined for c\`adl\`ag functions as well.

The paper is organized as follows. In Section \ref{sec:levy} we show that for a vector $X$ of functions a particular version of control, which we will call self-control, is sufficient for the pathwise existence of the L\'evy areas. An example of two functions is given which are not mutually controlled and for which consequently L\'evy's area fails to exist. In Section \ref{sec:foll} we study the question how control concepts and the existence of different kinds of integrals (It\^o type, Stratonovich type) are related, and in particular in which way control leads to versions of F\"ollmer's pathwise It\^o formula. Finally, provided the quadratic variation exists, we present a pathwise version of a functional It\^o's formula in Section \ref{sec:functional}.

\section{L\'evy's area and controlled paths}\label{sec:levy}

It is well-known that both the control of a path $Y$ with respect to another path $X$, as well as the existence of L\'evy's area for $X$ entails the existence of the rough path integral of $Y$ with respect to $X$ (Lyons \cite{Lyons1998}). This raises the question about the relative power of the hypotheses leading to the existence of the integral. This question will be answered here. We will show that control entails the existence of L\'evy's area. The analysis we present, as usual, is based on $d$-dimensional irregular paths, and corresponding notion of areas. For a continuous path $X\colon [0,T] \to \R^d$, say $X= (X^1,\dots, X^d)^*$, we recall that \textit{L\'evy's area} $\mathbb{L}(X)=(\mathbb{L}^{i,j}(X))_{i,j}$ is given by
\begin{align*}
  \mathbb{L}(X)^{i,j}:= \int_0^T X_t^i \dd X_t^j - \int_0^T X_t^j \dd X_t^i , \quad 1 \leq i,j \leq d,
\end{align*}
where $X^*$ denotes the transpose of the vector $X$, if the respective integrals exist. There are pairs of H\"older continuous paths $X^1$ and $X^2$ for which L\'evy's area does not exist (see Example \ref{ex:levy} below). To answer this question, we need the basic setup of rough path analysis, starting with the notion of power variation.

A \textit{partition} $\pi := \{ [t_{i-1}, t_i]\,:\, i=1,\dots,N\}$ of an interval $[0,T]$ is a family of essentially disjoint intervals such that $\bigcup_{i=1}^N [t_{i-1},t_i]=[0,T]$. For any $1 \leq p < \infty$, a continuous function $X \colon [0,T] \to \R^d$ is of finite $p$\textit{-variation} if
\begin{equation*}
  \vert\vert X \vert\vert_{p} := \sup_{\pi\in \mathcal{P}} \bigg ( \sum_{[s,t] \in \pi} \vert X_{s,t} \vert^p \bigg )^\frac{1}{p} < \infty,
\end{equation*}
where the supremum is taken over the set $\mathcal{P}$ of all partitions of $[0,T]$ and $X_{s,t}:= X_t-X_s$ for $s,t\in [0,T]$, $s\le t$. We write $\mathcal{V}^p([0,T],\R^d)$ for the set (linear space) of continuous functions of finite $p$-variation. Let, more generally, $R \colon [0,T]^2 \to \R^{d\times d}$ be a continuous function. In this case we consider the functional
\begin{align*}
  \vert \vert R \vert\vert_r := \sup_{\pi\in \mathcal{P}} \bigg ( \sum_{[s,t] \in \pi} \vert R_{s,t} \vert^{r} \bigg )^\frac{1}{r}, \quad 1\le r<\infty.
\end{align*}
An equivalent way to characterize the property of finite $p$-variation is by the existence of a control function. Denoting by $\Delta_T := \{ (s,t) \in [0,T]^2 \, : \, 0 \leq s \leq t \leq T \}$, we call a continuous function $\omega \colon \Delta_T \to \mathbb{R}^+$ vanishing on the diagonal \textit{control function} if it is superadditive, i.e. if for $(s,u,t)\in[0,T]^3$ one has $\omega(s,u) + \omega (u,t) \leq \omega(s,t)$ for $\quad 0 \leq s \leq u \leq t \leq T$.
Note that a function is of finite $p$-variation if and only if there exists a control function $\omega$ such that $\vert X_{s,t} \vert^p \leq \omega (s,t)$ for $(s,t)\in \Delta_T$. For a more detailed discussion of $p$-variation and control functions see Chapter 1.2 in \cite{Lyons2007}. For later reference we remark that all objects are analogously defined for general Banach spaces instead of $\R^d$. 

A fundamental insight due to Gubinelli \cite{Gubinelli2004} was that an integral $\int Y \dd X$ exists if ``$Y$ looks like $X$ in the small scale''. This leads to the concept of controlled paths, which we recall in its general form.

\begin{definition}
  Let $p,q,r \in \R^+$ be such that $2/p + 1/q > 1$ and $1/r = 1/p + 1/q$. Suppose $X \in \mathcal{V}^{p}([0,T],\R^d)$. We call $Y\in \mathcal{V}^{p}([0,T],\R^d)$ \textit{controlled} by $X$ if there exists $Y'\in \mathcal{V}^{q} ([0,T],\R^{d\times d})$ such that the remainder term $R^Y$ given by the relation $Y_{s,t}= Y_s^\prime X_{s,t} + R^Y_{s,t}$ satisfies $\vert\vert R^Y \vert \vert_{r} < \infty$. In this case we write $Y \in \CC_X^{q}$, and call $Y'$ \emph{Gubinelli derivative}.
\end{definition}

See Theorem 1 in \cite{Gubinelli2004} for the case of H\"older continuous paths, or Theorem 4.9 in \cite{Perkowski2013} for precise existence results of $\int Y \dd X$. Let us now modify this concept to a notion of control of a path by itself.

\begin{definition}
  Let $p,q,r \in \R^+$ be such that $2/p + 1/q > 1$ and $1/r = 1/p + 1/q$. We call $X \in \mathcal{V}^{p}([0,T],\R^d)$ \textit{self-controlled} if we have $X^i \in \CC_{X^j}^{q}$ or $X^j \in \CC_{X^i}^{q}$ for all $1\leq i, j \leq d$ with $i\neq j$.
\end{definition}

With this notion we are now able to deal with the main task of this section, the construction of the L\'evy area of a self-controlled path $X$. In fact, the integrals arising in L\'evy's area will be obtained via left-point Riemann sums as
\begin{equation}\label{eq:levy1}
  \mathbb{L}(X)^{i,j}= \int_0^T X^i_t \dd X^j_t - \int_0^T X_t^j \dd X^i_t := \lim_{\vert \pi \vert \to 0} \sum_{[s,t] \in \pi} ( X_s^i X^j_{s,t} - X_s^j X^i_{s,t}),
\end{equation}
for $1\leq i,j \leq d$, where $|\pi|$ denotes the mesh of a partition $\pi$. Our approach uses the abstract version of classical ideas due to Young \cite{Young1936} comprised in the so-called \textit{sewing lemma}.

\begin{lemma}\label{lem:young}[Corollary 2.3, Corollary 2.4 in \cite{Feyel2006}]
  Let $\Xi\colon \Delta_T \to \mathbb{R}^d$ be a continuous function and $K>0$ some constant. Assume that there exist a control function $\omega$ and a constant $\theta >1$ such that for all $(s,u,t)\in[0,T]^3$ with $0\leq s \leq u \leq t \leq T$ we have
  \begin{equation}\label{eq:young}
    \vert \Xi_{s,t} - \Xi_{s,u}- \Xi_{u,t}\vert\leq K \omega(s,t)^\theta.
  \end{equation}
  Then there exists a unique function $\Phi\colon [0,T] \to \R^d $ such that $\Phi(0)=0$ and
  \begin{equation*}
    \vert \Phi(t)-\Phi(s)- \Xi_{s,t}  \vert \leq C(\theta)\omega(s,t)^\theta \quad \text{and }\lim_{\vert \pi(s,t) \vert \to 0} \sum_{[u,v] \in \pi(s,t)}  \Xi_{u,v} = \Phi(t)-\Phi(s),
  \end{equation*}
  for $(s,t)\in \Delta_T$, where $C(\theta):= K (1-2^{1-\theta})^{-1}$ and $\pi(s,t)$ denotes a partition of $[s,t]$.
\end{lemma}

\begin{remark}
  For simplicity we state Lemma \ref{lem:young} only for a continuous function $\Xi \colon \Delta_T \to \R^d$. Yet, it still holds true without the continuity assumption and for a general Banach space replacing $\R^d$. See Theorem 1 and Remark 3 in \cite{Feyel2008}. Consequently, all results of this section extend to general Banach spaces.
\end{remark}

With this tool we now derive the existence of L\'evy's area for self-controlled paths of finite $p$-variation with $p\ge 1$.

\begin{theorem}\label{thm:levy1}
   Let $1 \leq p < \infty$ and suppose that $X \in \mathcal{V}^{p}([0,T],\R^d)$ is self-controlled, then L\'evy's area as defined in \eqref{eq:levy1} exists.
\end{theorem}

\begin{proof}
 Let $X \in \mathcal{V}^p([0,T],\R^{d})$ for $1 \leq p < \infty$ be self-controlled and fix $1\leq i,j\leq d$, $i\not= j$. We may assume without loss of generality that $X^i \in \CC^q_{X^j}$, i.e. $X^i_{s,t}= X_s^\prime(i,j) X^j_{s,t} + R^{i,j}_{s,t}$ and $\vert\vert X^\prime(i,j) \vert \vert_{q} ,\vert\vert R^{i,j} \vert \vert_{r} < \infty$. In order to apply Lemma \ref{lem:young}, we set $\Xi_{s,t}^{i,j}:=  X^i_s X^j_{s,t} - X^j_s X^i_{s,t}$ for $(s,t)\in \Delta_T$ and observe that for $(s,u,t)\in[0,T]^3$ with $0\leq s \leq u \leq t \leq T$, we have
  \begin{align*}
    \Xi^{i,j}_{s,t} - \Xi^{i,j}_{s,u}- \Xi^{i,j}_{u,t}
      &= X^j_{s,u}X^i_{u,t}-X^i_{s,u} X^j_{u,t}\\
      &= X^j_{s,u}(X_u^\prime(i,j) X^j_{u,t}+ R^{i,j}_{u,t})-(X_s^\prime(i,j) X^j_{s,u}+ R^{i,j}_{s,u}) X^j_{u,t} \\
      &= X^j_{s,u}R^{i,j}_{u,t} - R^{i,j}_{s,u} X^j_{u,t} + ( X_u^\prime(i,j)-X_s^\prime(i,j)) X^j_{s,u}X^j_{u,t}.
  \end{align*}
  Since the finite sum of control functions is again a control function, we can choose the same control function $\omega$ for $X^j, X^\prime(i,j)$ and $R^{i,j}$, and setting $\theta := \frac{2}{p}+\frac{1}{q} >1$ we get
  \begin{align*}
    \vert \Xi_{s,t}^{i,j} - \Xi_{s,u}^{i,j}- \Xi_{u,t}^{i,j}\vert \leq \omega(s,t)^{\frac{1}{p}+\frac{1}{r}} + \omega(s,t)^{\frac{1}{p}+\frac{1}{r}} + \omega(s,t)^{\frac{2}{p}+\frac{1}{q}}
    \leq 3 \omega (s,t)^\theta.
  \end{align*}
\end{proof}

We will next show that Riemann sums with arbitrary choices of base points for the integrand functions lead to the same L\'evy area as just constructed.

\begin{lemma}
  Let $X \in \mathcal{V}^{p}([0,T],\R^d)$ for some $1 \leq p < \infty$. Suppose that $X$ is self-controlled. Denote by $s'\in[s,t]$ an arbitrary point chosen in a partition interval $[s,t]\in\pi$. Then L\'evy's area from the preceding theorem is also given by
  \begin{equation*}
    \mathbb{L}(X)^{i,j} = \lim_{\vert \pi\vert \to 0} \sum_{[s,t] \in \pi} ( X^i_{s^\pr} X^j_{s,t} -  X^j_{s^\pr} X^i_{s,t}), \quad 1 \leq i,j \leq d.
  \end{equation*}
\end{lemma}

\begin{proof}
  For a self-controlled path $X \in \mathcal{V}^p([0,T],\R^{d})$ with $1 \leq p < \infty$ we may assume without loss of generality that $X^i \in \CC^q_{X^j}$ for $1\leq i,j\leq d$,   $i\not= j$.
  From Theorem \ref{thm:levy1} we already know that the left-point Riemann sums converge. Hence, we only need to show that
  \begin{equation}\label{eq:ass1}
    \sum_{[s,t] \in \pi_n} ( X^i_{s} X^j_{s,t} -  X^j_{s} X^i_{s,t}) - \sum_{[s,t] \in \pi_n} ( X^i_{s^\pr} X^j_{s,t} - X^j_{s^\pr} X^i_{s,t})
  \end{equation}
  tends to zero along every sequence of partitions $(\pi_n)$ such that the mesh $\vert \pi_n \vert$ converges to zero. Indeed, we may write for a partition interval $[s,t]$
  \begin{align*}
    X^i_{s} X^j_{s,t} -  X^j_{s} X^i_{s,t}& - ( X^i_{s^\pr} X^j_{s,t} -  X^j_{s^\pr} X^i_{s,t})
     = -X^i_{s,s^\pr} X^j_{s,t} + X^j_{s,s^\pr} X^i_{s,t}\\
    &= -(X^\pr_s(i,j) X^j_{s,s^\pr} +R^{i,j}_{s,s^\pr}) X^j_{s,t} + X^j_{s,s^\pr} ( X^\pr_s(i,j) X^j_{s,t} +R^{i,j}_{s,t})\\
    &= -R^{i,j}_{s,s^\pr}X^j_{s,t} + X^j_{s,s^\pr} R^{i,j}_{s,t}.
  \end{align*}
  Taking the same control function $\omega$ for $X^j$ and $R^{i,j}$, we estimate
  \begin{align*}
    \vert  X^i_{s} X^j_{s,t}  -  X^j_{s} X^i_{s,t} - ( X^i_{s^\pr} X^j_{s,t} -  X^j_{s^\pr} X^i_{s,t})\vert
      = \vert - R^{i,j}_{s,s^\pr}X^j_{s,t} + X^j_{s,s^\pr} R^{i,j}_{s,t} \vert
      \leq 2 \omega(s,t)^\theta
  \end{align*}
  with $\theta :=\frac{2}{p}+\frac{1}{p}> 1$. Recalling the superadditivity of $\omega$, we get for $n\in\mathbb{N}$
  \begin{align*}
    \bigg \vert  \sum_{[s,t] \in \pi_n} ( X^i_{s,s^\pr} X^j_{s,t} -  X^j_{s,s^\pr} X^i_{s,t})\bigg \vert
     \leq \sum_{[s,t] \in \pi_n}  \omega(s,t)^\theta
     \leq  \max_{[s,t] \in \pi_n} \omega(s,t)^{\theta-1} \omega(0,T),
  \end{align*}
  which means that \eqref{eq:ass1} tends to zero as $n\to\infty$.
\end{proof}

\begin{example}
  Let $(B_t\,;\, t\in [0,T])$ be a standard Brownian motion on a probability space $(\Omega, \mathcal{F},\mathbb{P})$ and let $f\in C^1(\mathbb{R},\mathbb{R})$ be a continuously differentiable function with $\alpha$-H\"older continuous derivative for $\alpha >0$. The trajectories of $B$ are of finite $p$-variation for all $p>2$ outside a null set $\mathcal{N}$. Thus we can deduce from Theorem \ref{thm:levy1} that L\'evy's area of $(B + g_1,f(B)+ g_2)$ exists outside the same null set $\mathcal{N}$ whenever $g_1,g_2 \in \mathcal{V}^{q}([0,T],\R^d)$ for some $1 \leq q < 2$.
\end{example}

The following example illustrates that for $p\geq 2$ things are essentially different. It will in particular show that in this case self-control of a path is necessary for the existence of L\'evy's area.
\begin{example}\label{ex:levy}
  Let us consider for $m\in \mathbb{N}$ the functions $X^m\colon [-1,1] \to \mathbb{R}^2$ with components given by
  \begin{equation*}
    X^{1,m}_t := \sum^m_{k=1} a_k \sin (2^k \pi t) \quad \text{ and } \quad X^{2,m}_t :=  \sum^m_{k=1} a_k \cos (2^k \pi t),\quad t \in [-1,1],
  \end{equation*}
  where $a_k := 2^{-\alpha k}$ and $\alpha \in (0,1)$. Set $X:=\lim_{m\to \infty} X^m$. These functions are $\alpha$-H\"older continuous uniformly in $m$. Indeed, let $s,t \in [-1,1]$ and choose $k\in \mathbb{N}$ such that $2^{-k-1}\leq \vert s-t \vert \leq 2^{-k}$. Then we can estimate as follows 
  \begin{align*}
   \vert X^{1,m}_t - X^{1,m}_s  \vert
    & = \bigg \vert \sum_{l=1}^m a_l 2 \cos (2^{l-1} \pi (s+t)) \sin (2^{l-1} \pi (s-t) ) \bigg \vert \\
    & \leq 2 \sum_{l=1}^k \vert a_l\vert \vert \sin (2^{l-1} \pi (s-t) )\vert  + 2 \sum_{l=k+1}^\infty \vert a_l \vert\\
    & \leq 2 \sum_{l=1}^k \vert a_l \vert 2^{l-1} \pi\vert s-t\vert + 2 \sum_{l=k+1}^\infty \vert a_l \vert  \\
    & \leq  \sum_{l=1}^k 2^{l-\alpha l}\pi \vert s-t\vert + 2^{-\alpha (k+1)+1} \frac{1}{1-2^{-\alpha}}\\
    & \leq \frac{2^{(k+1)(1-\alpha)}-1}{2^{1-\alpha}-1} \pi \vert s-t \vert  + \frac{2^{1- \alpha}}{1-2^{-\alpha}} \vert s- t \vert^\alpha  \\
    & \leq \frac{2^{(k+1)(1-\alpha)}-1}{2^{1-\alpha}-1} \pi 2^{-k(1-\alpha)} \vert s-t \vert^\alpha  + \frac{2^{1- \alpha}}{1-2^{-\alpha}} \vert s- t \vert^\alpha
     \leq C \vert s - t\vert^\alpha
  \end{align*}
  for some constant $C >0$ independent of $m \in \mathbb{N}$. Analogously, we can get the $\alpha$-H\"older continuity of $X^{2,m}$. Furthermore, it can be seen with the same estimate that $(X^m)$ converges uniformly to $X$ and thus also in $\alpha$-H\"older topology. The limit function $X$ is not $\beta$-H\"older continuous for every $\beta >\alpha$. In order to see this, choose $s=0$ and $t=t_n = 2^{-n}$ for $n \in \mathbb{N}$ and observe that
  \begin{equation*}
    \frac{\vert X^1_{t_n} - X^1_0 \vert}{\vert t_n - 0 \vert^{\beta}} = \sum_{k=1}^{n-1} 2^{-\alpha k + \beta n} \sin (2^{k-n} \pi) \geq 2^{(\beta - \alpha)n+\alpha},
  \end{equation*}
  which obviously tends to infinity as $n$ tends to infinity. Since $\alpha$-H\"older continuity is obviously related to finite $\frac{1}{\alpha}$-variation, we can conclude that $X\in \mathcal{V}^{\frac{1}{\alpha}}([-1,1],\mathbb{R}^2),$ and $X\not\in \mathcal{V}^{\gamma}([-1,1],\mathbb{R}^2)$ for $\gamma<\frac{1}{\alpha}.$  Let us now show that $X$ possesses no L\'evy area. For this purpose, fix $\alpha\in (0,1)$ and $m \in \mathbb{N}$. Then L\'evy's area for $X^m$ is given by 
  \begin{align*}
    \int_{-1}^1 X^{1,m}_s & \dd X^{2,m}_s -  \int_{-1}^1 X^{2,m}_s \dd X^{1,m}_s \\
    =& - \sum_{k,l=1}^m a_k a_l \int_{-1}^1 \big( \sin ( 2^k \pi s )  \sin( 2^l \pi s ) 2^l \pi + \cos (2^l \pi s) \cos (2^k \pi s) 2^k \pi \big ) \dd s \\
    =& - \sum_{k,l=1}^m  a_k a_l \bigg ( 2^l \pi \int_{-1}^1 \frac{1}{2} \big( \cos ((2^k-2^l) \pi s)-\cos ((2^k+2^l) \pi s)\big )\dd s\\
                                             &\hspace{2.5cm} +  2^k \pi \int_{-1}^1\frac{1}{2} \big(\cos ((2^k-2^l) \pi s ) + \cos ((2^k+2^l) \pi s) \big)\dd s\bigg )\\
    =& - 2 \sum_{k=1}^m a^2_k 2^k \pi
    = - 2 \sum_{k=1}^m 2^{(1-2\alpha)k} \pi.
  \end{align*}
  This quantity diverges as $m$ tends to infinity for $\frac{1}{\alpha} \ge 2$. Since $(X^m)$ converges to $X$ in the $\alpha$-H\"older topology, we can use this result to choose partition sequences of $[-1,1]$ along which Riemann sums approximating the L\'evy area of $X$ diverge as well. This shows that $X$ possesses no L\'evy area. In return Theorem \ref{thm:levy1} implies that $X$ cannot be self-controlled. However, it is not to hard to see directly that no regularity is gained by controlling $X^1$ with $X^2$. For this purpose, note that for $-1\le s\le t\le 1$, and $0\not=X_s'\in\mathbb{R}$, one has
  \begin{align*}
    \vert  X_{s,t}^1- X_s^\pr  X^2_{s,t} \vert & =  \bigg \vert \sum_{k=1}^\infty a_k \big [(\sin(2^k \pi t) - \sin(2^k \pi s))-X_s^\pr(\cos (2^k \pi t) - \cos(2^k \pi s)) \big ] \bigg\vert \\
    =\bigg \vert 2 \sum_{k=1}^\infty a_k &\big[  \sin(2^{k-1} \pi (s-t)) \cos(2^{k-1} \pi (s+t))\\
     &\hspace{1.9cm}+X_s^\pr \sin (2^{k-1} \pi (s+t))\sin(2^{k-1} \pi (s-t))\big] \bigg\vert \\
    =\bigg \vert 2 \sum_{k=1}^\infty a_k & \sin(2^{k-1} \pi (s-t)) \sqrt{1+(X_s^\pr)^2} \sin (2^{k-1} \pi (s+t) + \arctan ((X_s^\pr)^{-1}) )\bigg\vert.
  \end{align*}
  Let us now investigate H\"older regularity at $s=0$. First, assume $X_0'>0$, and take $t=2^{-n}$ to obtain
  \begin{align*}
    &\frac{\vert X_{0,2^n}^1- X_0' X^2_{0,2^n} \vert }{2^{-\beta n}} \\
    &\qquad =2^{\beta n}\bigg \vert 2 \sum_{k=1}^{n} a_k \sin(2^{k-1-n} \pi) \sqrt{1+(X_0')^2} \sin (2^{k-1-n} \pi + \arctan ((X_0')^{-1}) )\bigg\vert \\
    &\qquad \geq 2^{(\beta - \alpha )n} \sin\big( \frac{\pi}{2} + \arctan ((X_0')^{-1})\big).
   \end{align*}
   For $X_0' < 0$ the same estimates work for $t_n=-2^{-n} $ instead. Therefore, the H\"older regularity at $0$ cannot be better than $\alpha$ and in particular $X$ cannot be self-controlled for $\frac{1}{\alpha} > 2$.
\end{example}

\section{F\"ollmer integration}\label{sec:foll}

In his seminal paper F\"ollmer \cite{Follmer1979} considered one dimensional pathwise integrals. He was able to give a pathwise meaning to the limit
\begin{equation*}
  \int_0^T \DD F(X_t)\dd^{\pi_n} X_t := \lim_{n \to \infty} \sum_{[s,t]\in \pi_n} \langle \DD F(X_s), X_{s,t}\rangle,
\end{equation*}
provided $F \in C^2(\R^d,\R)$. A translation of F\"ollmer's work, today named \textit{F\"ollmer integration}, can be found in the appendix of \cite{Sondermann2006}. His starting point was the hypothesis that quadratic variation of $X \in C([0,T],\R^d)$ exists along a sequence of partitions $(\pi_n)_{n\in\mathbb{N}}$ whose mesh tends to zero. Here $\langle \cdot,\cdot \rangle$ denotes the usual inner product on $\R^d$. As indicated and discussed below, this construction of an integral depends strongly on the chosen sequence of partitions $(\pi_n)_{n\in\mathbb{N}}$.

Before coming back to an approach to F\"ollmer's integral, we shall construct a Stratonovich type integral, thereby discussing the problem of dependence on a chosen sequence of partitions. As in the previous section, our approach is based on the notion of controlled paths. This will also lead us on a route which does not require the existence of iterated integrals as in the classical rough path approach. We fix a $\gamma\in[0,1]$, to discuss Stratonovich limits for Riemann sums where integrands are taken as convex combinations $\gamma Y_s + (1-\gamma) Y_t$ of the values of $Y$ at the extremes of a partition interval $[s,t]$. We start by decomposing these sums into symmetric and antisymmetric parts. For $p,q \in [1,\infty)$, $X \in \mathcal{V}^p([0,T],\mathbb{R}^d)$ and $Y\in \CC^q_X$ we have
\begin{align}\label{eq:sa}
  \gint_0^T  &  Y_t \dd X_t := \lim_{\vert \pi \vert \to 0}  \sum_{[s,t] \in \pi} \langle Y_s + \gamma Y_{s,t}, X_{s,t}\rangle \nonumber \\
  = & \frac{1}{2} \bigg ( \gint_0^T  Y_t \dd X_t  + \gint_0^T X_t \dd Y_t \bigg ) + \frac{1}{2}\bigg (\gint_0^T  Y_t \dd X_t- \gint_0^T  X_t\dd Y_t \bigg ) \nonumber\\
  =: & \frac{1}{2} \mathbb{S}_\gamma\langle X,Y \rangle + \frac{1}{2} \mathbb{A}_\gamma\langle X,Y \rangle.
\end{align}
Note that $\gamma=0$ corresponds to the classical It\^o integral and $\gamma = \frac{1}{2}$ to the classical Stratonovich integral. 

If the variation orders of $X$ and $Y$ fulfill  $1/p + 1/q >1$, we are in the framework of Young's integration theory. Below $1$, either the existence of the rough path or control is needed. To illustrate this, we go back to Example \ref{ex:levy}.

\begin{example}
  Let $X = (X^1, X^2)$ be given according to Example \ref{ex:levy}. In this case, we have seen that $X^1$ and $X^2$ are of finite $\frac{1}{\alpha}$-variation. With decomposition \eqref{eq:sa} we see that
  \begin{align*}
    \frac{1}{2}\text{-}\int_0^1  &  X^2_t \dd X^1_t := \lim_{\vert \pi \vert \to 0}  \sum_{[s,t] \in \pi} \langle X^2_s + \frac{1}{2} X^2_{s,t}, X^1_{s,t}\rangle
    =  \frac{1}{2} \mathbb{S}_\frac{1}{2}\langle X^1,X^2 \rangle + \frac{1}{2} \mathbb{A}_\frac{1}{2}\langle X^1,X^2 \rangle\\
    = &\frac{1}{2} \lim_{\vert \pi \vert \to 0}\sum_{[s,t]\in\pi} \frac{1}{2} \big (\langle X^2_s+X^2_t, X^1_t-X^1_s\rangle + \langle X^1_s+X^1_t, X^2_t-X^2_s\rangle\big) + \frac{1}{2} \mathbb{L}^{1,2}(X)\\
    = & \frac{1}{2} \lim_{\vert \pi \vert \to 0} \sum_{[s,t]\in\pi} \langle X^1,X^2\rangle_{s,t} + \frac{1}{2} \mathbb{L}^{1,2}(X)
    = \frac{1}{2}( X^1_1 X^2_1 - X^1_0 X^2_0 ) + \frac{1}{2} \mathbb{L}^{1,2}(X),
  \end{align*}
  provided all terms are well-defined. Therefore, the integral exists if and only if L\'evy's area exists, which is not the case for instance if $\alpha=\frac{1}{2}.$ So beyond Young's theory, the existence of the $\frac{1}{2}$-Stratonovich integral is closely linked to the existence of L\'evy's area.
\end{example}

Using a suitable control concept, we will next construct the Stratonovich integral described above, but not just with restriction to a particular sequence of partitions. This time, the symmetry of the Gubinelli derivative of a controlled path plays an essential role. However, this symmetry assumption can be avoided if the involved paths control each other.

\begin{definition}
  Let $X,Y \in \mathcal{V}^{p}([0,T],\R^d)$. We say that $X$ and $Y$ are \textit{similar} if there exist $X^\prime,Y^\prime\in \mathcal{V}^{q} ([0,T],\R^{d\times d})$ such that $X\in \CC_Y^{q} $ with Gubinelli derivative $X'$, $Y\in \CC_X^{q}$ with Gubinelli derivative $Y'$, and $((X_t^\pr)^*)^{-1} = Y_t^\pr$ for all $t \in [0,T]$. In this case we write $Y \in \SP^{q}_X$.
\end{definition}

Let us give a very simple example of two paths $X,Y \in \mathcal{V}^{p}([0,T],\R^d)$ such that $Y\in \SP^{q}_X$ but neither $Y\in \CC_X^{q}$ with $Y^\pr$ symmetric nor $X\in \CC_Y^{q}$ with $X^\pr$ symmetric.

\begin{example}
  For $p \in [2,3)$ take $X^1 \in \mathcal{V}^{p}([0,T],\R)$ and $X^2,X^3 \in \mathcal{V}^{\frac{p}{2}}([0,T],\R)$. If we set $X:=(X^1,X^2,X^3)$ and $Y:=(X^1,0,0)$, we obviously have $X,Y \in \mathcal{V}^{p}([0,T],\R^3)$. In this case we could choose $X^\pr$ and $Y^\pr$ identical to $(z_1,z_2,z_3)$, where $z^*_1 := (1,0,0)$, $z^*_2 := (0,0,1)$, and $z^*_3 := (0,-1,0)$. We see that $Y \in \SP_X^{p}$, but $X^\pr$ and $Y^\pr$ are not symmetric matrices.
\end{example}

Under both assumptions we prove the existence of the Stratonovich integral described above. This time, thanks to the additional requirements of the Gubinelli derivative, the usual concept of controlled paths is sufficient, and L\'evy's area is not needed.

\begin{theorem}\label{thm:levy2}
  Let $\gamma\in[0,1]$, $X \in \mathcal{V}^p ([0,T],\R^d)$. If $ Y \in \CC_X^{q}$ and $Y^\prime_t$ is a symmetric matrix for all $t \in [0,T]$, then the antisymmetric part
  \begin{equation}\label{eq:levy}
    \mathbb{A}_\gamma\langle X,Y \rangle := \lim_{\vert \pi \vert \to 0} \sum_{[s,t] \in \pi} \big(\langle Y_s+\gamma Y_{s,t}, X_{s,t}\rangle - \langle X_s+\gamma X_{s,t}, Y_{s,t}\rangle\big),
  \end{equation}
  exists and satisfies
  \begin{equation*}
    \mathbb{A}_\gamma\langle X,Y \rangle = \mathbb{A}\langle X,Y \rangle := \lim_{\vert \pi \vert \to 0} \sum_{[s,t] \in \pi} \big(\langle Y_{s^\pr}, X_{s,t}\rangle - \langle X_{s^\pr}, Y_{s,t}\rangle\big)
  \end{equation*}
  for every choice of points $s^\pr\in[s,t] \in \pi$.
  The same result holds if $Y \in \SP^q_X$.
\end{theorem}

\begin{proof}
  It is easy to verify that by definition the antisymmetric part, if it exists as a limit of the Riemann sums considered, has to satisfy the second formula of the claim at least with the choice $s' = s$, for all intervals $[s,t]$ belonging to a partition. To prove that this limit exists, we use Lemma \ref{lem:young}. For this purpose, we set $\Xi_{s,t}:= \langle Y_s, X_{s,t}\rangle - \langle X_s, Y_{s,t}\rangle$ for $(s,t) \in \Delta_T$. Since $Y$ is controlled by $X$, we obtain
  \begin{align*}
    \Xi_{s,t} - \Xi_{s,u}- \Xi_{u,t}
      &= \langle Y_u^\prime X_{u,t}+ R^Y_{u,t}, X_{s,u}\rangle -\langle X_{u,t},Y_s^\prime X_{s,u}+ R^Y_{s,u}\rangle\\
      &= \langle R^Y_{u,t}, X_{s,u}\rangle - \langle X_{u,t},R^Y_{s,u}\rangle\ + \langle Y_u^\prime X_{u,t}, X_{s,u}\rangle - \langle X_{u,t},Y_s^\prime X_{s,u}\rangle\\
      &= \langle R^Y_{u,t}, X_{s,u}\rangle - \langle X_{u,t},R^Y_{s,u}\rangle\ + \langle  X_{u,t},Y_u^\prime X_{s,u} - Y_s^\prime X_{s,u}\rangle
  \end{align*}
  for $0\leq s < u< t \leq T$, where we used $\langle Y_u^\prime X_{u,t}, X_{s,u}\rangle = \langle  X_{u,t},Y_u^\prime X_{s,u}\rangle$ in the last line thanks to symmetry. With the same control $\omega$ for all functions involved as above, this gives
  \begin{align*}
    \vert \Xi_{s,t} - \Xi_{s,u}- \Xi_{u,t}\vert \leq \omega(s,t)^{\frac{1}{p}+\frac{1}{r}} + \omega(s,t)^{\frac{1}{p}+\frac{1}{r}} + \omega(s,t)^{\frac{2}{p}+\frac{1}{q}}
    \leq 3 \omega (s,t)^\theta
  \end{align*}
  with $\theta := \frac{2}{p}+\frac{1}{q} >1$. So from  Lemma \ref{lem:young} we conclude that the left-point Riemann sums converge. It remains to show that
  \begin{equation}\label{eq:ass}
   \sum_{[s,t] \in \pi_n} \big ( \langle Y_{s^\pr}, X_{s,t}\rangle -  \langle X_{s^\pr} , Y_{s,t}\rangle\big)-  \sum_{[s,t] \in \pi_n} \big( \langle Y_{s}, X_{s,t} \rangle- \langle X_{s}, Y_{s,t}\rangle \big) 
  \end{equation}
  tends to zero along every sequence of partitions $(\pi_n)$ such that the mesh $\vert \pi_n \vert$ converges to zero. Applying the symmetry of $Y^\pr$, we get
  \begin{align*}
    \langle Y_{s'}, X_{s,t} \rangle - \langle X_{s'}, Y_{s,t} \rangle &- \big ( \langle Y_{s} , X_{s,t}\rangle - \langle X_{s}, Y_{s,t}\rangle\big)\\
    &= \langle Y^\pr_s X_{s,s^\pr} + R^Y_{s,s^\pr},X_{s,t} \rangle - \langle X_{s,s^\pr},  Y^\pr_s X_{s,t} + R^Y_{s,t} \rangle\\
    &= \langle R^Y_{s,s^\pr},X_{s,t} \rangle - \langle X_{s,s^\pr}, R^Y_{s,t} \rangle,
  \end{align*}
  and thus
  \begin{equation*}
    \big \vert \langle Y_{s}, X_{s,t}\rangle  -  \langle X_{s}, Y_{s,t}\rangle - \big( \langle Y_{s^\pr}, X_{s,t}\rangle - \langle X_{s^\pr}, Y_{s,t}\rangle \big)\big\vert \leq \omega(s,t)^\theta
  \end{equation*}
  with $\theta :=\frac{1}{p}+\frac{1}{r}> 1$, where we choose the same control function $\omega$ for $X$ and $R^Y$. Therefore, the properties of $\omega$ imply
  \begin{align*}
    \bigg \vert  \sum_{[s,t] \in \pi_n} \big ( \langle Y_{s,s^\pr}, X_{s,t}\rangle -  \langle X_{s,s^\pr}, Y_{s,t}\rangle \big )\bigg \vert
     \leq  \sum_{[s,t] \in \pi_n}  \omega(s,t)^\theta
     \leq  \max_{[s,t] \in \pi_n} \omega(s,t)^{\theta-1} \omega(0,T),
  \end{align*}
  which means that \eqref{eq:ass} tends to zero as $\vert \pi_n \vert$ tends to zero.\\
  If we instead assume, that $X$ and $Y$ are similar, we obtain
  \begin{align*}
    &\Xi_{s,t} - \Xi_{s,u}- \Xi_{u,t}
      = \langle X^\pr_s Y_{s,u} + R^X_{s,u}, Y^\pr_u X_{u,t} + R^Y_{u,t} \rangle - \langle Y_{s,u}, X_{u,t}\rangle \\
      &= \langle X^\pr_s Y_{s,u} , R^Y_{u,t} \rangle +\langle R^X_{s,u}, Y^\pr_u X_{u,t} \rangle+\langle R^X_{s,u},R^Y_{u,t} \rangle+ \langle X^\pr_s Y_{s,u}, Y^\pr_u X_{u,t}\rangle -\langle Y_{s,u}, X_{u,t}\rangle
  \end{align*}
  for $0\leq s \leq u \leq t \leq T$. The last two terms in the preceding formula can be rewritten as
  \begin{align*}
    \langle X^\pr_s Y_{s,u}, Y^\pr_u X_{u,t}\rangle-\langle Y_{s,u}, X_{u,t}\rangle
       &=\langle Y_{s,u},(X_t^\pr)^* Y^\pr_u X_{u,t}- X_{u,t}\rangle\\
       &= \langle Y_{s,u},(X_t^\pr)^* ( Y^\pr_u -  Y_t^\pr ) X_{u,t}\rangle.
  \end{align*}
  Here we applied $((X_t^\pr)^*)^{-1} = Y_t^\pr$. Since the finite sum of control functions is again a control function, we can choose the same control function $\omega$ for $X,X^\prime,R^X$ and $Y,Y^\prime, R^Y$, and obtain
  \begin{align*}
    \vert \Xi_{s,t} &- \Xi_{s,u}  - \Xi_{u,t}\vert\\
    &\leq  \vert\vert  X^\pr \vert\vert_\infty  \omega (s,t)^{\frac{1}{p}+\frac{1}{r}} + \vert\vert Y^\pr\vert\vert_\infty \omega (s,t)^{\frac{1}{r}+ \frac{1}{p}} + \omega (s,t)^{\frac{1}{r}+\frac{1}{r}} + \vert\vert X^\pr \vert\vert_\infty  \omega (s,t)^{\frac{1}{q}+\frac{2}{p}}\\
    &\leq \big( 2 \vert\vert  X^\pr \vert\vert_\infty +\vert\vert Y^\pr\vert\vert_\infty + \omega(0,T)^{\frac{1}{q}}\big ) \omega (s,t)^\theta,
  \end{align*}
  where $\vert\vert \cdot \vert\vert_\infty $ denotes the supremum norm and $\theta := 2/p + 1/q  > 1$.
  We therefore have shown that the left-point Riemann sums converge. It remains to prove that \eqref{eq:ass} goes to zero along every sequence of partitions $(\pi_n)$ such that the mesh $\vert \pi_n \vert$ tends to zero. Since $X$ and $Y$ are similar, we observe that for $(s,t)\in\Delta_T$, and $s'\in[s,t]$
  \begin{align*}
   \langle Y_{s^\pr} , X_{s,t}\rangle - \langle X_{s^\pr}, Y_{s,t}\rangle-& \big ( \langle  Y_{s}, X_{s,t} \rangle -  \langle X_{s}, Y_{s,t} \rangle\big) \\
    &= \langle Y^\pr_s X_{s,s^\pr}, X_{s}^\pr Y_{s,t}\rangle +\langle Y^\pr_s X_{s,s^\pr}, R^X_{s,t}\rangle +\langle R^Y_{s,s^\pr}, X^\pr_s Y_{s,t}\rangle\\
    &\qquad+ \langle R^Y_{s,s^\pr}, R^X_{s,t} \rangle- \langle Y_{s,t}, X_{s,s^\pr}\rangle\\
    &= \langle Y^\pr_s X_{s,s^\pr}, R^X_{s,t}\rangle +\langle R^Y_{s,s^\pr}, X^\pr_s Y_{s,t}\rangle + \langle R^Y_{s,s^\pr} , R^X_{s,t} \rangle.
  \end{align*}
  To obtain the last line, we once again use $((X_s^\pr)^*)^{-1} = Y_s^\pr$. Taking again the same control function $\omega$ for $X,X^\pr, R^X$ and $Y,Y^\pr,R^Y$, we estimate
  \begin{align*}
    \big \vert \langle Y_{s}, X_{s,t}\rangle & - \langle X_{s}, Y_{s,t}\rangle - \big( \langle Y_{s^\pr}, X_{s,t}\rangle - \langle X_{s^\pr}, Y_{s,t}\rangle \big)\big\vert
      \leq C \omega(s,t)^\theta,
  \end{align*}
  where $C:= \vert\vert X^\pr\vert\vert_\infty + \vert \vert Y^\pr \vert\vert_\infty+\omega (0,T)^{1/q}$ with $\theta :=\frac{2}{p}+\frac{1}{p}> 1$. Superadditivity of $\omega$ finally gives
  \begin{align*}
    \bigg \vert  \sum_{[s,t] \in \pi_n} \big ( \langle Y_{s,s^\pr}, X_{s,t}\rangle -  \langle X_{s,s^\pr}, Y_{s,t}\rangle & \big )\bigg \vert
    \leq C  \omega(0,T) \max_{[s,t] \in \pi_n} \omega(s,t)^{\theta-1} ,
  \end{align*}
  which means that \eqref{eq:ass} tends to zero as $\vert \pi_n \vert$ tends to zero.
\end{proof}

\begin{remark}
  The proof of Theorem \ref{thm:levy2} works analogously under the assumption that $X$ is controlled by $Y$ and $X^\prime_t$ is a symmetric matrix for all $t \in [0,T]$. Moreover, if  $Y$ is controlled by $X$ and $Y^\prime_t$ is an antisymmetric matrix for all $t \in [0,T]$, then an analogous result to Theorem \ref{thm:levy2} holds true for the symmetric part $\mathbb{S}_\gamma\langle X,Y \rangle$.
\end{remark}

In case $\gamma = \frac{1}{2}$ as in the example above, the symmetric part simplifies considerably, and therefore the preceding theorem will already imply the existence of the  $\frac{1}{2}$-Stratonovich integral.

\begin{corollary}\label{cor:Stra}
  Let $X \in \mathcal{V}^p ([0,T],\R^d)$, $ Y \in\CC_X^{q}$ and suppose $Y^\prime_t$ is a symmetric matrix for all $t \in [0,T]$ or $Y \in \SP^q_X $.
  Then, the Stratonovich integral
  \begin{equation}\label{eq:Stra}
    \int_0^T Y_t\sd X_t :=  \lim_{\vert \pi \vert \to 0} \sum_{[s,t] \in \pi} \langle Y_s+ \frac{1}{2} Y_{s,t}, X_{s,t}\rangle
  \end{equation}
  exists and satisfies
  \begin{equation*}
    \frac{1}{2}\text{-}\int_0^T Y_t \dd X_t = \int_0^T Y_t \circ \dd X_t  = \frac{1}{2}\big(\langle Y_T, X_{T}\rangle - \langle Y_0, X_0 \rangle\big) +\frac{1}{2} \mathbb{A} \langle X, Y\rangle.
  \end{equation*}
\end{corollary}

\begin{proof}
   By equation \eqref{eq:sa} we may separately treat the symmetric part $\mathbb{S}_\frac{1}{2}\langle X,Y\rangle$ and the antisymmetric part $\mathbb{A}_\frac{1}{2}\langle X,Y\rangle$ of the integral $ \frac{1}{2}\text{-}\int_0^T  Y_t \dd X_t $. The existence of the antisymmetric part $\mathbb{A}_\frac{1}{2}\langle X,Y\rangle$ follows from Theorem \ref{thm:levy2}. For the symmetric part, note that as in Example \ref{ex:levy}
  \begin{align*}
    \langle Y_s  +  \frac{1}{2} Y_{s,t}, X_{s,t}\rangle + \langle X_s+ \frac{1}{2} X_{s,t}, Y_{s,t} \rangle
     =  \langle Y, X \rangle_{s,t}, \quad (s,t)\in\Delta_T.
  \end{align*}
  Therefore, $\mathbb{S}_\frac{1}{2}\langle X,Y\rangle $ is given by
  \begin{align}\label{eq:sym1}
    \mathbb{S}_\frac{1}{2}\langle X,Y\rangle &= \lim_{\vert \pi\vert \to 0}\sum_{[s,t] \in \pi} \big(\langle Y_s +  \frac{1}{2} Y_{s,t}, X_{s,t}\rangle + \langle X_s+ \frac{1}{2} X_{s,t}, Y_{s,t}\rangle \big)\nonumber\\
    & = \langle Y_T, X_T\rangle- \langle X_0, Y_0\rangle.
  \end{align}
  The proof works analogously for $Y \in \SP^q_X$.
\end{proof}

The discussion of $\gamma$-Stratonovich integrals above has shown that the corresponding antisymmetric component can be treated by means of the concept of path control. In the case $\gamma \not= \frac{1}{2}$, a symmetric term is left to consider. This does not seem to be possible by means of the ideas used for the antisymmetric component. And this brings us back to F\"ollmer's approach. Our treatment of the symmetric part reflects the role played by \emph{quadratic variation} in F\"ollmer's approach, and will therefore be strongly dependent on partition sequences. For this purpose we define the quadratic variation in the sense of F\"ollmer (cf. \cite{Follmer1979}), and call a sequence of partitions $(\pi_n)$ \textit{increasing} if for all $[s,t] \in \pi_n$ there exist $ [t_i, t_{i+1}] \in \pi_{n+1}$, $i=1,\dots,N$, such that $[s,t] = \bigcup_{i=1}^N [t_i,t_{i+1}]$.

\begin{definition}
  Let $(\pi_n)$ be an increasing sequence of partitions such that \\ $\lim_{n \to \infty} \vert \pi_n\vert=0$.
  A continuous function $f \colon [0,T] \to \mathbb{R}$ has \textit{quadratic variation} along $(\pi_n)$ if the sequence of discrete measures on $([0,T], \mathcal{B}([0,T]))$ given by
  \begin{equation}\label{eq:follmer}
    \mu_n := \sum_{[s,t] \in \pi_n} \vert f_{s,t} \vert^2 \delta_{s}
  \end{equation}
  converges weakly to a measure $\mu$, where $\delta_s$ denotes the Dirac measure at $s \in [0,T]$. We write $[ f ]_t$ for the ``distribution function'' of the interval measure associated with $\mu$. A continuous path $X = (X^1,\dots,X^d)$ has \textit{quadratic variation} along $(\pi_n)$ if \eqref{eq:follmer} holds for all $X^i$ and $X^i+X^j$, $1\leq i,j \leq d$. In this case, we set
  \begin{equation*}
    [X^i,X^j]_t := \frac{1}{2}\big([X^i+X^j]_t - [X^i]_t - [X^j]_t\big), \quad t \in [0,T].
  \end{equation*}
\end{definition}

\begin{remark}
  Since in our situation the limiting distribution function is continuous, weak convergence is equivalent to uniform convergence to the distribution function. Hence, $X = (X^1,\dots,X^d)  \in C([0,T], \R^d)$ has quadratic variation in the sense of F\"ollmer if and only if
  \begin{equation*}
    [ X^i,X^j ]_t^n := \sum_{[u,v]\in \pi_n} X_{u \wedge t, v \wedge t}^i X_{u \wedge t, v \wedge t}^j
  \end{equation*}
  converges uniformly to $[X^i,X^j]$ in $C([0,T],\R)$ for all $1\leq i,j \leq d$, where $u\wedge t := \min \{u,t\}$. See Lemma 4.20 in \cite{Perkowski2013}.
\end{remark}

\begin{remark}
  Let us emphasize here that \emph{quadratic variation} should not be confused with the notion of \emph{$2$-variation}: quadratic variation depends on the choice of a partition sequence $(\pi_n)$, $2$-variation does not. In fact, for every continuous function $f \in C([0,T],\mathbb{R})$ there exits a sequence of partitions $(\pi_n)$ with $\lim_{n \to \infty} \vert \pi_n\vert=0$ such that $[ f, f ]_t=0$ for all $t\in [0,T]$. See for instance Proposition 70 in \cite{Freedman1983}.
\end{remark}

The existence of quadratic variation guaranteed, F\"ollmer was able to prove a pathwise version of It\^o's formula. In his case, the construction of the integral is closely linked to the partition sequence chosen for the quadratic variation. We will now aim at combining the techniques of controlled paths with the quadratic variation hypothesis, and derive a pathwise version of It\^o's formula for paths with finite quadratic variation, in which the quadratic variation term may depend on a partition sequence, but the integral does not. As a first step, we derive the existence of $\gamma$-Stratonovich integrals for any $\gamma\in[0,1].$ To do so, we will need the following technical lemma, the easy proof of which is left to the reader.

\begin{lemma}\label{lem:quad}
  Let $p\ge 1$, $(\pi_n)$ be an increasing sequence of partitions such that $\lim_{n \to \infty} \vert \pi_n\vert=0$, $X\in \mathcal{V}^p([0,T],\R^d)$ with quadratic variation along $(\pi_n)$ and $Y \in \CC^q_X $. In this case the \emph{quadratic covariation} of $X$ and $Y$ exists and is given by
  \begin{equation*}
    [ Y,X ]_T := \lim_{n \to \infty} \sum_{[s,t] \in \pi_n}\langle X_{s,t},Y_{s,t}\rangle = \sum_{1 \leq i,j \leq d} \int_0^T Y^\pr_t (i,j) \dd^{\pi_n} [X^i,X^j]_t,
  \end{equation*}
  where $Y_t^\pr = (Y_t^\pr(i,j))_{1\leq i,j\leq d}$, for $0\le t\le T$.
\end{lemma}

\begin{theorem}\label{thm:foll}
  Let $X \in \mathcal{V}^p ([0,T],\R^d)$, $ Y \in\CC_X^{q}$ and suppose $Y^\prime_t$ is a symmetric matrix for all $t \in [0,T]$ or $Y \in \SP^q_X $. Let $(\pi_n)$ be an increasing sequence of partitions such that $\lim_{n \to \infty}\vert \pi_n \vert = 0$ and $X$ has quadratic variation along $(\pi_n)$. Then for all $\gamma \in [0,1]$ the $\gint Y_t \dd^{\pi_n} X_t$ integral exists and is given by
  \begin{equation*}
    \gint_0^T Y_t \dd^{\pi_n} X_t =  \int_0^T  Y_t \circ \dd X_t  + \frac{1}{2} (2 \gamma -1) \sum_{1 \leq i,j \leq d}  \int_0^T Y^\pr_t (i,j) \dd^{\pi_n} [X^i,X^j]_t,
  \end{equation*}
  where $Y_t^\pr = (Y_t^\pr(i,j))_{1\leq i,j\leq d}$.
\end{theorem}

\begin{proof}
  Fix $\gamma \in [0,1]$. As before we split the sum as in \eqref{eq:sa} into its symmetric and antisymmetric part:
  \begin{align*}
    \sum_{[s,t] \in \pi_n} \langle Y_s+ \gamma Y_{s,t}, X_{s,t}\rangle
      = & \frac{1}{2} \sum_{[s,t] \in \pi_n} \big(\langle Y_s+ \gamma Y_{s,t}, X_{s,t}\rangle+ \langle X_s+ \gamma X_{s,t}, Y_{s,t}\rangle \big) \\
        & + \frac{1}{2} \sum_{[s,t] \in \pi_n} \big(\langle Y_s+ \gamma Y_{s,t}, X_{s,t} \rangle - \langle X_s+ \gamma X_{s,t}, Y_{s,t}\rangle \big).
  \end{align*}
  The second sum converges for every sequence of partitions $(\pi_n)$ with $\lim_{n \to \infty} \vert \pi_n \vert = 0$ and is independent of $\gamma$ thanks to Theorem \ref{thm:levy2}. Taking $\gamma=1/2$ we can apply Corollary \ref{cor:Stra} to see that
  \begin{equation}\label{eq:sym2}
   \frac{1}{2} \mathbb{A}\langle X,Y\rangle = \int_0^T  Y_t \circ \dd X_t - \frac{1}{2}\big ( \langle X_{T}, Y_{T}\rangle-\langle X_{0},Y_{0}\rangle\big ).
  \end{equation}
  For the symmetric part, we note for $(s,t)\in\Delta_T$
  \begin{align*}
     \langle Y_s + \gamma Y_{s,t}, X_{s,t}\rangle+ \langle X_s+ \gamma X_{s,t}, Y_{s,t}\rangle 
     = & (1-\gamma) \big(\langle Y_t, X_t \rangle- \langle Y_s, X_s \rangle - \langle X_{s,t},Y_{s,t}\rangle\big)\\
      &+ \gamma \big( \langle Y_t, X_t\rangle - \langle Y_s, X_s \rangle +\langle X_{s,t}, Y_{s,t}\rangle\big) \\
     =&  \langle Y_t, X_t \rangle- \langle Y_s, X_s\rangle + (2\gamma-1)\langle X_{s,t},Y_{s,t}\rangle.
  \end{align*}
  Thus the first sum reduces to
  \begin{align*}
    \frac{1}{2}\sum_{[s,t] \in \pi_n}  \big(\langle Y_s+ \gamma Y_{s,t},& X_{s,t}\rangle +  \langle X_s+ \gamma X_{s,t}, Y_{s,t}\rangle\big) \\
     & =  \frac{1}{2} \big( \langle Y_T, X_T\rangle - \langle Y_0, X_0\rangle \big) +  \frac{2\gamma-1}{2} \sum_{[s,t] \in \pi_n} \langle X_{s,t}, Y_{s,t}\rangle.
  \end{align*}
  Therefore, the symmetric part converges along $(\pi_n)$, and the assertion follows by \eqref{eq:sym2} and Lemma \ref{lem:quad}.\\
  The statement for $Y \in \SP^q_X$ can be proven analogously.
\end{proof}

An application of Theorem \ref{thm:foll} to the particular case $Y = \DD F(X)$ for a smooth enough function $F$ provides the classical Stratonovich formula.

\begin{lemma}\label{lem:levy2}
  Let $1\leq p<3$, $X \in \mathcal{V}^p([0,T],\R^d)$ and $F \in C^2(\mathbb{R}^d,\mathbb{R})$. Suppose that the second derivative $\DD^2 F$ is $\alpha$-H\"older continuous of order $\alpha > \max\{p-2,0\}$. Then the Stratonovich integral $\int  \DD F(X_t) \circ \dd X_t $  exists and is given by
  \begin{equation*}
    \int_0^T \DD F(X_t) \circ \dd X_t = F(X_T) - F(X_0).
  \end{equation*}
\end{lemma}

\begin{proof}
  Let $X = (X^1, \dots , X^d)^*\in \mathcal{V}^p([0,T],\R^d)$ for $1 \leq p < 3$. Then, with $r=\frac{p}{2}$ in the definition of controlled paths we easily see that $\DD F(X) \in \CC^{p}_X$. Thus by Corollary \ref{cor:Stra} the ($\frac{1}{2}$-)Stratonovich integral is well-defined and independent of the chosen sequence of partitions $(\pi_n)$ along which the limit is taken. Now choose an increasing sequence of partitions $(\pi_n)$ such that $\lim_{n \to \infty} \vert \pi_n \vert = 0$ and $[ X ]_t = 0$ along $(\pi_n)$ for $t \in [0,T]$ (cf. Proposition 70 in \cite{Freedman1983}). Applying Taylor's theorem to $F$, we observe that
  \begin{align*}
    F(X_T) -& F( X_0) = \frac{1}{2}\sum_{[s,t] \in \pi_n} \big ( (F(X_t) - F(X_s)) -  (F(X_s) - F(X_t))\big ) \\
     =& \sum_{[s,t] \in \pi_n} \langle \frac{1}{2} \DD F(X_s) + \frac{1}{2} \DD F(X_t), X_{s,t} \rangle  +  \sum_{[s,t] \in \pi_n} ( R(X_s,X_t) + \tilde R(X_s,X_t) )\\
      &+ \frac{1}{4} \sum_{[s,t] \in \pi_n} \sum_{1 \leq i,j \leq d} (\DD^2_{i,j}(X_s)- \DD^2_{i,j}(X_t)) X_{s,t}^i X_{s,t}^j, \\
  \end{align*}
  where $\vert R(x,y)\vert +\vert \tilde R (x,y)\vert \leq \phi(\vert x - y \vert ) \vert x-y \vert^2,$ for some increasing function $\phi \colon [0,\infty) \to \R$ such that $\phi (c) \to 0$ as $c \to 0$. Since $X$ is continuous and has zero quadratic variation along $(\pi_n)$, the last two terms converge to $0$ as $n\to \infty$, and we obtain
  \begin{align*}
    \int_0^T \DD F(X_t)  \circ \dd X_t  &= \lim_{n \to \infty} \sum_{[s,t] \in \pi_n} \langle\DD F(X_s) + \frac{1}{2}( \DD F(X_t)-\DD F(X_s)), X_{s,t}\rangle\\
    &= F(X_T) - F(X_0).
  \end{align*}
\end{proof}

We can now present the announced version of the pathwise formula by F\"ollmer (cf. \cite{Follmer1979}), for which the proof reduces to combining the previous results of Theorem \ref{thm:foll} and Lemma \ref{lem:levy2}.

\begin{corollary}\label{cor:follmer}
  Let $1 \leq p < 3$, $\gamma \in [0,1]$ and $(\pi_n)$ be an increasing sequence of partitions such that $\lim_{n \to \infty} \vert \pi_n\vert=0$. Assume $F \in C^2(\mathbb{R}^d, \mathbb{R})$ with $\alpha$-H\"older continuous second derivative $\DD^2 F$ for some $\alpha > \max\{p-2,0\}$. If $X\in \mathcal{V}^p([0,T], \mathbb{R}^d)$ has quadratic variation along $(\pi_n)$, then the formula
  \begin{align*}
    F(X_T) =  F(X_0) &+ \gint_0^T \DD F(X_t) \dd^{\pi_n} X_t \\
             &- \frac{1}{2} (2 \gamma-1) \sum_{1 \leq i,j \leq d}  \int_0^T \DD_{i,j}^2 F(X_s) \dd^{\pi_n} [X^i,X^j]_s
  \end{align*}
  holds.
\end{corollary}

The assumptions, that $X$ is of finite $p$-variation for some $1 \leq p < 3$ and that the second derivative $\DD^2 F$ is $\alpha$-H\"older continuous for some $\alpha > \max\{p-2,0\}$ can be considered as the price we have to pay for obtaining an integral of which the antisymmetric part does not depend on the chosen partition sequence. F\"ollmer \cite{Follmer1979} does not need these hypotheses and especially not that the integrand is controlled by the integrator. This leads to a much bigger class of admissible integrands as we will see in the next subsection.

\subsection{Functional It\^o formula}\label{sec:functional}

In recent years, functional It\^o calculus which extends classical calculus to functionals depending on the whole path of a stochastic process and not only on its current value, has received much attention. Based on the notion of derivatives due to Dupire \cite{Dupire2009}, in a series of papers Cont and Fourni\'e \cite{Cont2010,Cont2010a,Cont2013} developed a functional It\^o formula. One drawback of their approach is that the involved functional has to be defined on the space of c\`adl\`ag functions, or at least on a subspace strictly larger than $C([0,T],\R^d)$ (see \cite{Cosso2014}), and not only on $C([0,T],\R^d)$. In the spirit of F\"ollmer the paper \cite{Cont2010a} provides a non-probabilistic version of a probabilistic It\^o formula shown in \cite{Cont2010,Cont2013}.

The present subsection takes reference to this program. We generalize F\"ollmer's pathwise It\^o formula (cf. \cite{Follmer1979} or Corollary \ref{cor:follmer}) to twice Fr\'echet differentiable functionals defined on the space of continuous functions. Our functional It\^o formula might be seen as the pathwise analogue to formulas stated in \cite{Ahn1997}.

First we have to fix some further notation. Let $(\pi_n)$ be an increasing sequence of partitions such that $\lim_{n \to \infty} \vert \pi_n\vert=0$ and $ X \in C([0,T],\R^d)$. We denote by $X^n$ the piecewise linear approximation of $X$ along $(\pi_n)$, i.e.
\begin{equation}\label{eq:approx2}
  X^n_t := \frac{X_{t^n_{j+1}}-X_{t^n_j}}{t_{j+1}^n-t_{j}^n}(t-t^n_j)+X_{t^n_j}, \quad t\in [t^n_j,t^n_{j+1}),\quad \text{for }[t^n_j,t^n_{j+1}]\in \pi_n.
\end{equation}
In the following $\mathcal{C}$ stands for $C([0,T],\R^d)$ and $\mathcal{C}^*$ for the dual space of $\mathcal{C}$. For $X\in \mathcal{C}$ we define $X^t_s:=X_s\1_{[0,t)}(s)+X_t\1_{[t,T]}(s)$ and $X^{n,t}_s:=X^{n}_s\1_{[0,t)}(s)+X^{n}_t\1_{[t,T]}(s)$ for $s \in [0,T]$, where $\1_{[t,T]}$ is the indicator function of the interval $[t,T]$. Assume $F\colon \mathcal{C}\to \R$ is twice continuously (Fr\'echet) differentiable. That is, $\DD F \colon \mathcal{C}\to \mathcal{C}^*$ and $\DD^2 F \colon \mathcal{C}\to \mathcal{L}(\mathcal{C},\mathcal{C}^*)$ are continuous with respect to the corresponding norms. It is well-known that $\mathcal{L}(\mathcal{C},\mathcal{C}^*)$ is isomorphic to $\mathcal{C}\otimes\mathcal{C}$. For each $t\in [0,T]$ we can understand $\1_{[t,T]}$ as an element of $\mathcal{C}^{**}$, the bidual of $\mathcal{C}$, and $\1_{[t,T]}\otimes\1_{[t,T]}$ as an element in $(\mathcal{C}\otimes\mathcal{C})^{**}$, respectively. Hence, $\langle\DD F(X^s),\1_{[s,T]}\rangle$ and $\langle\DD^2 F(X^s),\1_{[s,T]}\otimes \1_{[s,T]}\rangle$ are well-defined as dual pairs.

\begin{theorem}\label{thm:functionalIto}
  Let $(\pi_n)$ be an increasing sequence of partitions such that the mesh satisfies $\lim_{n \to \infty} \vert \pi_n\vert=0$, and $X\in \mathcal{C}$ with quadratic variation along $(\pi_n)$. Suppose $F\colon [0,T] \times \mathcal{C}\to \R$ is continuously differentiable with respect to the first argument and twice continuously differentiable with respect to the second. Furthermore, assume that $\partial_t F$ and $\DD^2 F$ are bounded and uniformly continuous. Then, for all $t \in [0,T]$ we have
  \begin{align}\label{eq:functionalIto}
    F(t,X^t)= &F(0,X^0) + \int_0^t \partial_t F(s,X^s) \dd s+\sum_{i=1}^d \int_0^t \langle\DD_i F(s,X^s),\1_{[s,T]}\rangle\dd^{\pi_n} X^i_s\nonumber \\
     &+ \frac{1}{2}\sum_{i,j=1}^d\int_0^t \langle\DD^2_{i,j}F(s,X^s),\1_{[s,T]}\otimes \1_{[s,T]}\rangle\dd [ X^i,X^j]_s,
  \end{align}
  where the integral is given by
  \begin{align*}
    \sum_{i=1}^d \int_0^t  \langle\DD_i F(s,X^s),\1_{[s,T]} &\rangle\dd^{\pi_n} X^i_s \\
     &:= \lim_{n\to \infty} \sum_{i=1}^d \sum_{[t^n_k,t^n_{k+1}]\in \pi_n(t)} \langle\DD_i F(t^n_k,X^{n,t^n_k}), \eta^n_{t^n_j}\rangle X^i_{t^n_k,t^n_{k+1}},
  \end{align*}
  where $\pi_n(t):=\{[u,v\wedge t]\,:\,[u,v]\in \pi_n,\, u<t \}$ and $\eta^n_{t^n_j}$ for $[t^n_j,t^n_{j+1}]\in \pi_n(t)$ by
  \begin{equation*}
    \eta^n_{t^n_j}(s):=\frac{(s \vee t^n_{k+1})-t^n_j}{t^n_{k+1}-t_k^n}\1_{[t^n_k,T]}(s),  \quad s \in [0,T].
  \end{equation*}
\end{theorem}

\begin{proof}
  To increase the readability of the proof, we assume $d=1$. The general result follows analogously. Let $t \in [0,T]$ and $(\pi_n)$ a sequence of partitions fulfilling the assumption of Theorem \ref{thm:functionalIto}. We easily see that
  \begin{align}\label{eq:telescop}
    F(t,X^{n,t}) - & F(0,X^{n,0})\nonumber\\
    &= \sum_{[t^n_k,t_{k+1}^n]\in \pi_n(t)} \big(F(t^n_{k+1},X^{n,t^n_{k+1}}) -F(t^n_k,X^{n,t^n_{k+1}})\nonumber\\
    &\qquad\qquad\qquad\qquad\qquad +  F(t^n_{k},X^{n,t^n_{k+1}})-F(t^n_{k},X^{n,t^n_k})\big)
  \end{align}
  and note that the right hand side converges uniformly to $F(t,X^{t}) - F(0,X^{0})$ as $n\to \infty$. Applying a Taylor expansion, we obtain
  \begin{equation*}
    F(t^n_{k+1},X^{n,t^n_{k+1}}) -F(t^n_k,X^{n,t^n_{k+1}})= \partial_t F(t^n_k,X^{n,t^n_{k+1}}) (t^n_{k+1}- t^n_{k}) + R(t^n_{k},t^n_{k+1}),
  \end{equation*}
  where one has $|R(t^n_{k},t^n_{k+1})|\leq \phi_1(|t^n_{k+1}-t^n_{k}|)|t^n_{k+1}-t^n_{k}|$, for some continuous function $\phi_1 \colon [0,\infty)\to \mathbb{R}$ such that $\phi_1 (c)\to 0$ as $c\to 0$. With this observation and the continuity of $\partial_t F(s,X^s)$, we conclude by dominated convergence that
  \begin{equation*}
    \lim_{n \to \infty }\sum_{[t^n_k,t_{k+1}^n]\in \pi_n(t)} \big(F(t^n_{k+1},X^{n,t^n_{k+1}})-F(t^n_k,X^{n,t^n_{k+1}}) \big)=\int_0^t \partial_t F(s,X^s) \dd s.
  \end{equation*}
  For the second difference of equation \eqref{eq:telescop}, we use a second order Taylor expansion to get
  \begin{align*}
   &\sum_{[t^n_k,t_{k+1}^n]\in \pi_n(t)} F(t^n_{k},X^{n,t^n_{k+1}})-F(t^n_{k},X^{n,t^n_k})\\
    &= \sum_{[t^n_k,t_{k+1}^n]\in \pi_n(t)}\langle\DD F(t^n_{k},X^{n,t^n_k}),X^{n,t^n_{k+1}}-X^{n,t^n_k}\rangle\\
    &\quad+ \sum_{[t^n_k,t_{k+1}^n]\in \pi_n(t)}\frac{1}{2}\langle\DD^2 F(t^n_{k},X^{n,t^n_k}),(X^{n,t^n_{k+1}}-X^{n,t^n_k})\otimes (X^{n,t^n_{k+1}}-X^{n,t^n_k})\rangle\\
    &\quad+\sum_{[t^n_k,t_{k+1}^n]\in \pi_n(t)} \tilde R(X^{n,t^n_k},X^{n,t^n_{k+1}})=:S_n^1(t)+S_n^2(t)+S_n^3(t),
  \end{align*}
  where $|\tilde R(X^{n,t^n_k},X^{n,t^n_{k+1}}) |\leq \phi_2 (\|X^{n,t^n_{k+1}}-X^{n,t^n_{k}}\|_\infty)\|X^{n,t^n_{k+1}}-X^{n,t^n_k}\|_{\infty}^2$, for some continuous function $\phi_2 \colon \mathbb{R}\to \mathbb{R}$ such that $\phi_2(c) \to 0$ as $c \to 0$.
  Since $X^{n,t^n_{k+1}}-X^{n,t^n_k}= \eta^n_{t^n_j} X_{t^n_k,t^n_{k+1}}$ and $[\cdot, \cdot]$ is bilinear, $S_n^1$ and $S_n^2$ can be rewritten by
  \begin{align*}
    S^1_n(t) &= \sum_{[t^n_k,t_{k+1}^n]\in \pi_n(t)} \langle\DD F(t^n_{k},X^{n,t^n_k}),\eta^n_{t^n_j}\rangle X_{t^n_k,t^n_{k+1}}, \\
    S^2_n(t) &= \sum_{[t^n_k,t_{k+1}^n]\in \pi_n(t)} \langle\DD^2 F(t^n_{k},X^{n,t^n_k}),\eta^n_{t^n_j}\otimes\eta^n_{t^n_j}\rangle X_{t^n_k,t^n_{k+1}}^2,\\
  \end{align*}
  and $S_n^3$ estimated by
  \begin{equation*}
    \sup_{t\in [0,T]}|S^3_n(t)|\leq \max_{[t^n_k,t_{k+1}^n]\in \pi_n(t)} \phi_2(|X_{t^n_k,t^n_{k+1}}|) \sum_{[t^n_k,t_{k+1}^n]\in \pi_n(t)}  X_{t^n_k,t^n_{k+1}}^2.
  \end{equation*}
  Because $X$ has quadratic variation along $(\pi_n)$ and $\phi_2 (|X_{t^n_k,t^n_{k+1}}|) \to 0$ as $n\to \infty$, $S^3_n (\cdot)$ tends uniformly to zero. To see the convergence of $S_n^2(t)$, we set $\lambda_n(s):=\max \{t^n_j \,:\, [t_j^n,t_{j+1}^n]\in \pi_n,\, t^n_j\leq s\}$ and define
  \begin{align*}
    &f_n(s):= \langle\DD^2 F(\lambda_{n(s)}, X^{n,\lambda_n(s)}), \eta^n_{\lambda_n(s)} \otimes \eta^n_{\lambda_n(s)}\rangle,\quad \text{and}\\
    &f(s):=\langle\DD^2 F(s,X^s),\1_{[s,T]}\otimes \1_{[s,T]}\rangle,\quad s \in [0,T].
  \end{align*}
  Note that $(f_n)$ is a sequence of left-continuous functions which are uniformly bounded in $n$. Additionally, $\lim_{n \to \infty} f_n(s)=f(s)$ for each $s\in [0,T]$ as
  \begin{align*}
   \lim_{n\to \infty} |f_n(s)- &f(s)|\leq   \lim_{n\to \infty} \big| \langle\DD^2 F(\lambda_{n(s)}, X^{n,\lambda_n(s)}), \eta^n_{\lambda_n(s)} \otimes \eta^n_{\lambda_n(s)} - \1_{[s,T]}\otimes \1_{[s,T]}\rangle\big| \\
     &+ \lim_{n\to \infty}  \big| \langle\DD^2 F(\lambda_{n(s)}, X^{n,\lambda_n(s)})-\DD^2 F(s, X^{s}),  \1_{[s,T]}\otimes \1_{[s,T]}\rangle\big|=0.
  \end{align*}
  The first summand tends to zero by weak convergence of $\eta^n_{\lambda_n(s)} \otimes \eta^n_{\lambda_n(s)}$ to $\1_{[s,T]}\otimes \1_{[s,T]}$, and the second one by Lemma 3.2 in \cite{Ahn1997}. By Proposition 2.1 in \cite{Ahn1997} $f$ is also left-continuous and so Lemma 12 in \cite{Cont2010a} implies
  \begin{equation*}
    \lim_{n \to \infty} S^2_n(t)= \int_0^t \langle\DD^2 F(s,X^s),\1_{[s,T]}\otimes \1_{[s,T]}\rangle\dd [X]_s.
  \end{equation*}
  In summary, we derived equation \eqref{eq:functionalIto} and implicitly the convergence of $S^1_n(t)$.
\end{proof}

It is fairly easy to see that $\langle\DD F(t,X^t),\1_{[t,T]}\rangle$ is in general not controlled by a path increment of $X$, which we briefly illustrate by revisiting Example 2.3 in \cite{Ahn1997}. Especially, this explains why we cannot just rely on Theorem \ref{thm:foll} to prove Theorem \ref{thm:functionalIto}.

\begin{example}
  Let $\mu$ be a finite signed Borel measure and let $F\colon C([0,T],\R)\to \R$ be given by
  \begin{equation*}
    F(X):=\int_0^T g(s,X_s)\,\mu(\mathrm{d} s),
  \end{equation*}
  where $g(t,\cdot)\in C^2(\R,\R)$ for each $t \in [0,T]$ with bounded second partial derivatives $\DD_{x,x}^2 g$ and $g(\cdot,x)\colon [0,T]\to \R$ $\mu$-measurable. In this case $\langle\1_{[t,T]},\DD F(X^t)\rangle$ is of course in general not controlled by a path increment of $X$ as we see from the explicit calculation
  \begin{equation*}
     \langle \DD F(X^t), \1_{[t,T]} \rangle- \langle \DD F(X^s),\1_{[s,T]}\rangle  =- \int_s^t \DD_x g(u,X_u)\, \mu(\mathrm{d}u),\quad 0\leq s\leq t\leq T.
  \end{equation*}
\end{example}

\par\bigskip\noindent
{\bf Acknowledgment.} The authors are grateful to Randolf Altmeyer and Nicolas Perkowski for helpful comments and discussions on the subject matter. D.J.P. was financially supported by a Ph.D. scholarship of the DFG Research Training Group 1845 ``Stochastic Analysis with Applications in Biology, Finance and Physics''.

\bibliographystyle{amsplain}

\begin{thebibliography}{99}


\bibitem[1]{Ahn1997}
  Ahn, H.: Semimartingale integral representation, \emph{Ann. Probab.}
  \textbf{25} (1997), no.~2, 997--1010.

\bibitem[2]{Cont2010}
  Cont, R. and Fourni{\'e}, D.-A.: {A functional extension of the
  Ito formula},\emph{ C. R. Math. Acad. Sci. Paris} \textbf{348} (2010), no.~1-2,
  57--61.

\bibitem[3]{Cont2010a}
  Cont, R. and Fourni{\'e}, D.-A.: {C}hange of variable formulas for non-anti-cipative functionals on
  path space, \emph{J. Funct. Anal.} \textbf{259} (2010), no.~4, 1043--1072.

\bibitem[4]{Cont2013}
  Cont, R. and Fourni{\'e}, D.-A.: Functional {I}t\^o calculus and stochastic integral
  representation of martingales, \emph{Ann. Probab.} \textbf{41} (2013), no.~1,
  109--133.

\bibitem[5]{Cosso2014}
  Cosso, A. and Russo, F.: A regularization approach to functional
  {I}t\^o calculus and strong-viscosity solutions to path-dependent {PDE}s,
  \emph{Preprint} arXiv:1401.5034 (2014).

\bibitem[6]{Dupire2009}
  Dupire, B.: Functional {I}t{\^o} {C}alculus, \emph{Bloomberg Portfolio
  Research Paper} (2009), no.~2009-04-FRONTIERS, Available at SSRN: http://ssrn.com/abstract=1435551.

\bibitem[7]{Feyel2006}
  Feyel, D. and {De La Pradelle}, A.: {Curvilinear integrals along
  enriched paths}, \emph{{Electron. J. Probab.}} \textbf{11} (2006), 860--892.

\bibitem[8]{Feyel2008}
  Feyel, D., {De La Pradelle}, A. and Mokobodzki, G.: {A
  non-commutative sewing lemma}, \emph{{Electron. Commun. Probab.}} \textbf{13}
  (2008), 24--34.

\bibitem[9]{Friz2013}
  Friz, P. and Hairer, M.: \emph{{A {C}ourse on {R}ough {P}aths: {W}ith an
  {I}ntroduction to {R}egularity {S}tructures}}, Springer, 2014.

\bibitem[10]{Follmer1979}
  F{\"o}llmer, H.: {Calcul d'It{\^o} sans probabilit{\'e}s},
  \emph{S{\'e}minaire de Probabilit{\'e}s XV} \textbf{80} (1979), 143--150.

\bibitem[11]{Freedman1983}
  Freedman, D.: \emph{Brownian motion and diffusion}, second ed.,
  Springer-Verlag, New York, 1983.

\bibitem[12]{Gubinelli2013} 
  Gubinelli, M., Imkeller, P. and Perkowski, N.:
  {Paracontrolled distributions and singular PDEs}, to appear in \emph{Forum of Mathematics, Pi}, Preprint
  arXiv:1210.2684 (2013).

\bibitem[13]{Gubinelli2014}
  Gubinelli, M., Imkeller, P. and Perkowski, N.: {A Fourier Approach to pathwise stochastic integration},
  \emph{Preprint} arXiv:1410.4006 (2014).

\bibitem[14]{Gubinelli2004}
  Gubinelli, M.: {Controlling rough paths}, \emph{J. Funct. Anal.}
  \textbf{216} (2004), no.~1, 86--140.

\bibitem[15]{Lyons2007}
  Lyons, T.~J., Caruana, M. and  L{\'e}vy, T.: \emph{{Differential
  equations driven by rough paths}}, {Lecture Notes in Mathematics}, vol. 1908,
  Springer, Berlin, 2007.

\bibitem[16]{Lejay2009}
  Lejay, A.: \emph{Yet another introduction to rough paths}, S\'eminaire de
  {P}robabilit\'es {XLII}, Lecture Notes in Math., vol. 1979, Springer, Berlin,
  2009, pp.~1--101.

\bibitem[17]{Levy1940}
  L{\'e}vy, P.: {Le mouvement brownien plan}, \emph{Amer. J. Math.} \textbf{62}
  (1940), 487--550.

\bibitem[18]{Lyons1998}
  Lyons, T.~J.: {Differential equations driven by rough signals}, \emph{Rev.
  Mat. Iberoam.} \textbf{14} (1998), no.~2, 215--310.

\bibitem[19]{Perkowski2013a}
  Perkowski, N.S.: \emph{Studies of robustness in stochastic analysis and
  mathematical finance}, Ph.D. thesis, Humboldt-Universit\"at zu Berlin, 2014.

\bibitem[20]{Perkowski2013}
  Perkowski, N. and Pr{\"o}mel, D.~J.: {Pathwise stochastic integrals
  for model free finance}, \emph{Preprint} arXiv:1311.6187 (2013).

\bibitem[21]{Sondermann2006}
  Sondermann, D.: \emph{Introduction to stochastic calculus for finance: a new
  didactic approach}, vol. 579, Springer, 2006.

\bibitem[22]{Young1936}
  Young, L.~C.: {An inequality of the H{\"o}lder type, connected with
  Stieltjes integration}, \emph{Acta Math.} \textbf{67} (1936), no.~1, 251--282.

\end{thebibliography}

\end{document}